\documentclass[12pt]{amsart}

\usepackage{amsmath}
\usepackage{amssymb}
\usepackage{amsfonts}
\usepackage{amsthm}
\usepackage{enumerate}
\usepackage{hyperref}
\usepackage{color}

\theoremstyle{plain}
\newtheorem{thm}{Theorem}[section]

\newtheorem{prop}[thm]{Proposition}
\newtheorem{lem}[thm]{Lemma}
\newtheorem{cor}[thm]{Corollary}

\newtheorem{ques}[thm]{Question}

\newtheorem{conj}[thm]{Conjecture}

\theoremstyle{definition}
\newtheorem{dfn}[thm]{Definition}

\newtheorem{dfns-rems}[thm]{Definitions and Remarks}
\newtheorem{notas-rems}[thm]{Notations and Remarks}
\newtheorem{exmps-rems}[thm]{Examples and Remarks}

\begin{document}


\title[Stability of depth and Sdepth of symbolic powers]{Stability of depth and Stanley depth of symbolic powers of squarefree monomial ideals}


\author[S. A. Seyed Fakhari]{S. A. Seyed Fakhari}

\address{S. A. Seyed Fakhari, School of Mathematics, Statistics and Computer Science,
College of Science, University of Tehran, Tehran, Iran.}

\email{aminfakhari@ut.ac.ir}


\begin{abstract}
Let $\mathbb{K}$ be a field and $S=\mathbb{K}[x_1,\dots,x_n]$ be the
polynomial ring in $n$ variables over $\mathbb{K}$. Assume that $I\subset S$ is a squarefree monomial ideal. For every integer $k\geq 1$, we denote the $k$-th symbolic power of $I$ by $I^{(k)}$. Recently, Monta\~no and N\'u\~nez-Betancourt \cite{mn} proved that for every pair of integers $m, k\geq 1$,$${\rm depth}(S/I^{(m)})\leq {\rm depth}(S/I^{(\lceil\frac{m}{k}\rceil)}).$$We provide an alternative proof for this inequality. Moreover, we reprove the known results that the sequence $\{{\rm depth}(S/I^{(k)})\}_{k=1}^{\infty}$ is convergent and$$\min_k{\rm depth}(S/I^{(k)})=\lim_{k\rightarrow \infty}{\rm depth}(S/I^{(k)})=n-\ell_s(I),$$where $\ell_s(I)$ denotes the symbolic analytic spread of $I$. We also determine an upper bound for the index of depth stability of symbolic powers of $I$. Next, we consider the Stanley depth of symbolic powers and prove that the sequences $\{{\rm sdepth}(S/I^{(k)})\}_{k=1}^{\infty}$ and $\{{\rm sdepth}(I^{(k)})\}_{k=1}^{\infty}$ are convergent and the limit of each sequence is equal to its minimum. Furthermore, we determine an upper bound for the indices of sdepth stability of symbolic powers.
\end{abstract}






\maketitle


\section{Introduction} \label{sec1}

Let $\mathbb{K}$ be a field and $S=\mathbb{K}[x_1,\dots,x_n]$ be the
polynomial ring in $n$ variables over the field $\mathbb{K}$, and let
$I\subset S$ be a monomial ideal. The {\it analytic spread} of $I$, denoted by $\ell(I)$, is defined as the Krull dimension of $\mathcal{R}(I)/
{{\frak{m}}\mathcal{R}(I)}$, where $\mathcal{R}(I)=\bigoplus_
{k=0}^{\infty}I^k$ is the {\it Rees ring} of $I$ and
$\frak{m}=(x_1,\ldots,x_n)$ is the maximal ideal of $S$. A classical result by Burch \cite{b'} says
that $$\min_k{\rm depth}(S/I^k)\leq n-\ell(I).$$By a theorem of
Brodmann \cite{b}, ${\rm depth}(S/I^k)$ is constant for large $k$. We call
this constant value the {\it limit depth} of $I$, and denote it by
$\lim_{k\rightarrow \infty}{\rm depth}(S/I^k)$. Brodmann improved the Burch
inequality by showing that$$\lim_{k\rightarrow \infty}{\rm depth}(S/I^k)
\leq n-\ell(I).$$The smallest integer $t\geq 1$ such that ${\rm depth}(S/I^m)=\lim_{k\rightarrow \infty}{\rm depth}(S/I^k)$ for all $m\geq t$ is called the {\it index of depth stability of powers} of $I$ and is denoted by ${\rm dstab}(I)$. It is of great interest to compute the limit of the sequence $\{{\rm depth}(S/I^k)\}_{k=1}^{\infty}$ and to determine or bound its index of stability. The most general results in this direction were obtained in \cite{hh'', hq, t}. In \cite{hh''}, Herzog and Hibi proved that if the associated graded ring ${\rm gr}_I(S)$ is Cohen–Macaulay, then$$\lim_{t\rightarrow \infty}{\rm depth}(S/I^t)
= n-\ell(I).$$Herzog and Qureshi \cite{hq} showed that for every polymatroidal ideal $I$, we have ${\rm dstab}(I)\leq \ell(I)$. Furthermore, they asked whether it is true that for every squarefree monomial ideal $I$, the inequality ${\rm dstab}(I) < n$ holds. Trung \cite{t} investigated the case of edge ideals and proved that for any edge ideal $I(G)\subset S$, the limit $\lim_{k\rightarrow \infty}{\rm depth}(S/I(G)^k)$ is $n-\ell(I(G))$ which is equal to the number of bipartite connected components of $G$. Moreover, in the same paper, it is shown that for any graph $G$ with $n$ vertices, we have ${\rm dstab}(I(G))< n$. This gives a positive answer to the above mentioned question of Herzog and Qureshi, in the case of edge ideals.

It is also of interest to consider similar problems for the symbolic powers of monomial ideals. Let $I\subset S$ be a squarefree monomial ideal. For every integer $k\geq 1$, we denote the $k$-th symbolic power of $I$ by $I^{(k)}$. It immediately follows from \cite[Theorem 4.7]{ht} that the sequence $\{{\rm depth}(S/I^{(k)})\}_{k=1}^{\infty}$ is convergent. Let $\mathcal{R}_s(I)=\bigoplus_
{k=0}^{\infty}I^{(k)}$ be the {\it symbolic Rees ring} of $I$. The Krull dimension of $\mathcal{R}_s(I)/{{\frak{m}}\mathcal{R}(I)}$ is called the {\it symbolic analytic spread} of $I$ and is denoted by $\ell_s(I)$. Varbaro \cite[Proposition 2.4]{v1} proved that$$\min_k{\rm depth}(S/I^{(k)})=n-\ell_s(I).$$This equality was then improved in \cite{hktt} by showing that$$\min_k{\rm depth}(S/I^{(k)})=\lim_{k\rightarrow \infty}{\rm depth}(S/I^{(k)})=n-\ell_s(I).$$Let ${\rm dstab}_s(I)$ denote the {\it index of depth stability of symbolic powers} of $I$ which is the smallest integer $t\geq 1$ with ${\rm depth}(S/I^{(m)})=\lim_{k\rightarrow \infty}{\rm depth}(S/I^{(k)})$ for all $m\geq t$. In \cite{hktt}, it was also proven that$${\rm dstab}_s(I)\leq n(n+1){\rm bight}(I)^{n/2},$$where ${\rm bight}(I)$ is the maximum height of associated primes of $I$.

Recently, Monta\~no and N\'u\~nez-Betancourt \cite[Theorem 3.4]{mn} proved that for every squarefree monomial ideal and for any pair of integers $m, k\geq 1$, we have$${\rm depth}(S/I^{(m)})\leq {\rm depth}(S/I^{(\lceil\frac{m}{k}\rceil)}).$$In Theorem \ref{depsym}, we provide an alternative proof for the above inequality. While the proof in \cite{mn} is based on constructing a splittable map between distinct symbolic powers of $I$, our proof is based on a formula due to Takayama \cite{t1}. Next, we use this inequality to reprove that the sequence $\{{\rm depth}(S/I^{(k)})\}_{k=1}^{\infty}$ is convergent and$$\min_k{\rm depth}(S/I^{(k)})=\lim_{k\rightarrow \infty}{\rm depth}(S/I^{(k)}).$$Moreover, we provide an alternative proof for the equality$$\lim_{k\rightarrow \infty}{\rm depth}(S/I^{(k)})=n-\ell_s(I)$$(see Theorem \ref{dmain}). For every squarefree monomial ideal $I$, let ${\rm dmin}_s(I)$ denote the smallest integer $t\geq 1$ with ${\rm depth}(S/I^{(t)})=\lim_{k\rightarrow \infty}{\rm depth}(S/I^{(k)})$. In Theorem \ref{dmain}, we also determine an upper bound for ${\rm dstab}_s(I)$ in terms of ${\rm dmin}_s(I)$. More precisely, we show that$${\rm dstab}_s(I)\leq \max\{1, {\rm dmin}_s(I)^2-{\rm dmin}_s(I)\}.$$

Next, we study the Stanley depth of symbolic powers of squarefree monomial ideals. Let $M$ be a nonzero
finitely generated $\mathbb{Z}^n$-graded $S$-module. Let $u\in M$ be a
homogeneous element and $Z\subseteq \{x_1,\dots,x_n\}$. The $\mathbb
{K}$-subspace $u\mathbb{K}[Z]$ generated by all elements $uv$ with $v\in
\mathbb{K}[Z]$ is called a {\it Stanley space} of dimension $|Z|$, if it is
a free $\mathbb{K}[\mathbb{Z}]$-module. Here, as usual, $|Z|$ denotes the
number of elements of $Z$. A decomposition $\mathcal{D}$ of $M$ as a finite
direct sum of Stanley spaces is called a {\it Stanley decomposition} of
$M$. The minimum dimension of a Stanley space in $\mathcal{D}$ is called the
{\it Stanley depth} of $\mathcal{D}$ and is denoted by ${\rm
sdepth}(\mathcal {D})$. The quantity $${\rm sdepth}(M):=\max\big\{{\rm
sdepth}(\mathcal{D})\mid \mathcal{D}\ {\rm is\ a\ Stanley\ decomposition\
of}\ M\big\}$$ is called the {\it Stanley depth} of $M$. As a convention, we set ${\rm sdepth}(M)=\infty$, when $M$ is the zero module. For a reader friendly introduction to Stanley depth, we refer to \cite{psty} and for a nice survey on this topic, we refer to \cite{h}. Inspired by the limit behavior of depth of powers of ideals, Herzog \cite{h} proposed the following conjecture.

\begin{conj}
{\rm (}\cite[Conjecture 59]{h}{\rm )} For every monomial ideal $I$, the sequence $\{{\rm sdepth}(I^k)\}_{k=1}^{\infty}$ is convergent.
\end{conj}

One can ask the similar question for symbolic powers.

\begin{ques} \label{qmain}
Let $I\subset S$ be a squarefree monomial ideal $I$. Are the sequences $\{{\rm sdepth}(I^{(k)})\}_{k=1}^{\infty}$ and $\{{\rm sdepth}(S/I^{(k)})\}_{k=1}^{\infty}$ convergent?
\end{ques}

In \cite[Theorem 2.1]{s4}, we invented a method to compare the Stanley depth of monomial ideals (see Lemma \ref{scand}). Using this method, in Theorem \ref{sdepsym}, we show that if $I$ is a squarefree monomial ideal and $m$ and $k$ are positive integers, then for every integer $j$ with $m-k\leq j\leq m$, we have$${\rm sdepth}(I^{(m)})\geq {\rm sdepth}(I^{(km+j)}) \ \ \ \ \ \ \ {\rm and} \ \ \ \ \ \ \ {\rm sdepth}(S/I^{(m)})\geq {\rm sdepth}(S/I^{(km+j)}).$$

In Theorem \ref{sdmain}, we use the above inequalities to give a positive answer to Question \ref{qmain}. Furthermore, in the same theorem, we show that the equalities$$\min_k{\rm sdepth}(S/I^{(k)})=\lim_{k\rightarrow \infty}{\rm sdepth}(S/I^{(k)}),$$and$$\min_k{\rm sdepth}(I^{(k)})=\lim_{k\rightarrow \infty}{\rm sdepth}(I^{(k)})$$hold, for any squarefree monomial ideal. The assertions of Theorem \ref{sdmain} allow us to define the {\it indices of sdepth stability of symbolic powers} as follows.
$${\rm sdstab}_s(I)=\min \big\{t \mid {\rm sdepth}(I^{(m)})=\lim_{k\rightarrow \infty}{\rm sdepth}(I^{(k)}) \ {\rm for \ all} \ m\geq t\big\}$$
$${\rm sdstab}_s(S/I)=\min \big\{t \mid {\rm sdepth}(S/I^{(m)})=\lim_{k\rightarrow \infty}{\rm sdepth}(S/I^{(k)}) \ {\rm for \ all} \ m\geq t\big\}$$We also define the following quantities.
$${\rm sdmin}_s(I)=\min \big\{t \mid {\rm sdepth}(I^{(t)})=\lim_{k\rightarrow \infty}{\rm sdepth}(I^{(k)})\big\}$$
$${\rm sdmin}_s(S/I)=\min \big\{t \mid {\rm sdepth}(S/I^{(t)})=\lim_{k\rightarrow \infty}{\rm sdepth}(S/I^{(k)})$$In Corollary \ref{sstabmin}, we determine an upper bound for ${\rm sdstab}_s(I)$ (resp. ${\rm sdstab}_s(S/I))$ in terms of ${\rm sdmin}_s(I)$ (resp. ${\rm sdmin}_s(S/I)$). More precisely, we show that for every squarefree monomial ideal $I$, we have$${\rm sdstab}_s(I)\leq \max\{1, {\rm sdmin}_s(I)^2-{\rm sdmin}_s(I)\}$$and$${\rm sdstab}_s(S/I)\leq \max\{1, {\rm sdmin}_s(S/I)^2-{\rm sdmin}_s(S/I)\}.$$

In general, we do not know how to compute the limit values $\lim_{k\rightarrow \infty}{\rm sdepth}(I^{(k)})$ and $\lim_{k\rightarrow \infty}{\rm sdepth}(S/I^{(k)})$. However, we will see in Theorem \ref{mat} that if $I$ is the Stanley-Reisner ideal of matroid, then$$\lim_{k\rightarrow \infty}{\rm sdepth}(S/I^{(k)})=n-\ell_s(I)$$and$$\lim_{k\rightarrow \infty}{\rm sdepth}(I^{(k)})\geq n-\ell_s(I)+1.$$


\section{Preliminaries} \label{sec2}

In this section, we provide the definitions and basic facts which will be used in the next sections. We first recall the definition of symbolic powers which are the main objective of this paper.

\begin{dfn}
Let $I$ be an ideal of $S$ and let ${\rm Min}(I)$ denote the set of minimal primes of $I$. For every integer $k\geq 1$, the $k$-th {\it symbolic power} of $I$,
denoted by $I^{(k)}$, is defined to be$$I^{(k)}=\bigcap_{\frak{p}\in {\rm Min}(I)} {\rm Ker}(R\rightarrow (R/I^k)_{\frak{p}}).$$
\end{dfn}

Let $I$ be a squarefree monomial ideal in $S$ and suppose that $I$ has the irredundant
primary decomposition $$I=\frak{p}_1\cap\ldots\cap\frak{p}_r,$$ where each
$\frak{p}_i$ is a prime ideal generated by a subset of the variables of
$S$. It follows from \cite[Proposition 1.4.4]{hh'} that for every integer $k\geq 1$, $$I^{(k)}=\frak{p}_1^k\cap\ldots\cap
\frak{p}_r^k.$$

Assume that $I\subset S$ is an arbitrary ideal. An element $f \in S$ is
{\it integral} over $I$, if there exists an equation
$$f^k + c_1f^{k-1}+ \ldots + c_{k-1}f + c_k = 0 {\rm \ \ \ \ with} \ c_i\in I^i.$$
The set of elements $\overline{I}$ in $S$ which are integral over $I$ is the {\it integral closure}
of $I$. The ideal $I$ is {\it integrally closed}, if $I = \overline{I}$, and $I$ is {\it normal} if all powers
of $I$ are integrally closed. By \cite[Theorem 3.3.18]{v'}, a monomial ideal $I$ is normal if and only if the Rees ring $\mathcal{R}(I)$ is a normal ring.

Let $R$ and $R'$ be two commutative rings with identity such that $R'\subseteq R$. We say $R$ is an {\it integral extension} of $R'$, if for any element $r\in R$, there exists a monic polynomial $p(x)\in R'[x]$ with $p(r)=0$.

A {\it simplicial complex} $\Delta$ on the set of vertices $V(\Delta)=[n]:=\{1,
\ldots,n\}$ is a collection of subsets of $[n]$ which is closed under
taking subsets; that is, if $F \in \Delta$ and $F'\subseteq F$, then also
$F'\in\Delta$. Every element $F\in\Delta$ is called a {\it face} of
$\Delta$. The {\it dimension} of a face $F$ is defined to be $|F|-1$. The {\it dimension} of
$\Delta$ which is denoted by $\dim\Delta$, is defined to be $d-1$, where $d
=\max\{|F|\mid F\in\Delta\}$. A {\it facet} of $\Delta$ is a maximal face
of $\Delta$ with respect to inclusion. We say that $\Delta$ is {\it
pure} if all facets of $\Delta$ have the same cardinality. A {\it simplex} is a simplicial complex which has only one facet. The {\it link}
of $\Delta$ with respect to a face $F \in \Delta$, denoted by ${\rm lk_
{\Delta}}(F)$, is the simplicial complex$${\rm lk_{\Delta}}(F)=\{G
\subseteq [n]\setminus F\mid G\cup F\in \Delta\}$$and the {\it deletion} of
$F$, denoted by ${\rm del_{\Delta}}(F)$, is the simplicial complex$${\rm
del_{\Delta}}(F)=\{G \subseteq [n]\setminus F\mid G \in \Delta\}.$$When $F
= \{x\}$ is a single vertex, we abuse the notation and write ${\rm lk_{\Delta}}(x)$ and ${\rm del_{\Delta}}(x)$. The $i$-th reduced homology of $\Delta$ with coefficients in $\mathbb{K}$ will be denoted by $\widetilde{H}_i(\Delta; \mathbb{K})$.

Let $\Delta$ be a simplicial complex on $[n]$. The {\it Stanley-Reisner ideal} of $\Delta$ is defined as$$I_{\Delta}=\big(\prod_{i\in F}x_i : F\subseteq [n], F\notin \Delta\big)\subseteq S.$$A squarefree monomial ideal is {\it unmixed} if the associated prime ideals of $I$ have the same height. It follows from \cite[Lemma 1.5.4]{hh'} that the $I_{\Delta}$ is unmixed if and only if $\Delta$ is pure.

\begin{dfn}
Let $\Delta$ be a simplicial complex on the vertex set $[n]$. Then we say that $\Delta$ is {\it vertex decomposable} if either
\begin{itemize}
\item[(1)] $\Delta$ is a simplex, or\\[-0.3cm]

\item[(2)] there exists $x\in [n]$ such that ${\rm del_{\Delta}}(x)$ and
    ${\rm lk_{\Delta}}(x)$ are vertex decomposable and every facet of
    ${\rm del_{\Delta}}(x)$ is a facet of $\Delta$.
\end{itemize}
\end{dfn}

An interesting family of simplicial complexes is the class of matroids which is defined as follows.

\begin{dfn}
A simplicial complex $\Delta$ is called {\it matroid} if for every pair of faces $F, G \in \Delta$ with $|F| > |G|$, there is a vertex $x\in F\setminus G$ such that $G\cup\{x\}$ is a face of $\Delta$. It is well-know and easy to prove that every matroid is a pure simplicial complex.
\end{dfn}

We close this section by recalling the definition of clean modules which will be used in the proof of Theorem \ref{mat}.

Let $M$ be a finitely generated $\mathbb{Z}^n$-graded $S$-module and assume that ${\rm Min}(M)$ is the set of minimal primes of $M$. A chain$$\mathcal{F}: 0=M_0\subset M_1\subset\cdots\subset M_r=M$$of $\mathbb{Z}^n$-graded $S$-modules is a {\it prime filtration} of $M$, if for every integer $i=1, \ldots r$, there exists a monomial prime ideal $\frak{p}_i$ and a vector $\alpha_i\in \mathbb{Z}^n$ with $M_i/M_{i-1}\cong S/\frak{p}_i(-\alpha_i)$. The prime filtration $\mathcal{F}$ is called {\it clean} if ${\rm Min}(M)=\{\frak{p}_1, \ldots, \frak{p}_r\}$. The module $M$ is a {\it clean module} if it admits a clean prime filtration.


\section{Depth of symbolic powers} \label{sec3}

In this section, we study the depth of symbolic powers of squarefree monomial ideals. Our first goal is to reprove the result of Monta\~no and N\'u\~nez-Betancourt \cite[Theorem 3.4]{mn} which states that that for every squarefree monomial ideal $I$ and for any pair of integers $m, k\geq 1$, we have$${\rm depth}(S/I^{(m)})\leq {\rm depth}(S/I^{(\lceil\frac{m}{k}\rceil)}).$$Our proof is based on a formula due to Takayama \cite{t1}. Hence, we first recall this formula.

Let $I$ be a monomial ideal. As $S/I$ is a $\mathbb{Z}^n$-graded $S$-module, it follows that for every integer $i$, the local cohomology module $H_{\mathfrak{m}}^i(S/I)$ is $\mathbb{Z}^n$-graded too. For any vector $\alpha\in \mathbb{Z}^n$, we denote the $\alpha$-component of $H_{\mathfrak{m}}^i(S/I)$ by $H_{\mathfrak{m}}^i(S/I)_{\alpha}$. The {\it support} of the vector $\alpha=(\alpha_1, \ldots, \alpha_n)$ is defined to be the set ${\rm Supp}(\alpha)=\{i : \alpha_i > 0\}$ and its co-support is the set ${\rm CoSupp}(\alpha)=\{i : \alpha_i < 0\}$. For any subset $F\subseteq [n]$, let $S_F=S[x_i^{-1}: i\in F]$. For every vector $\alpha=(\alpha_1, \ldots, \alpha_n)\in \mathbb{Z}^n$, we set$$\Delta_{\alpha}(I)=\{F\subseteq [n]\setminus {\rm CoSupp}(\alpha) : \mathbf{x}^{\alpha}\notin IS_{F\cup {\rm CoSupp}(\alpha)}\},$$where $\mathbf{x}^{\alpha}=x_1^{\alpha_1}\ldots x_n^{\alpha_n}$. Takayama \cite[Theorem 2.2]{t1} proves that for any vector $\alpha\in \mathbb{Z}^n$ and for every integer $i$, we have
\[
\begin{array}{rl}
\dim_{\mathbb{K}} H_{\mathfrak{m}}^i(S/I)_{\alpha}=\dim_{\mathbb{K}} \widetilde{H}_{i-\mid{\rm CoSupp}(\alpha)\mid-1}(\Delta_{\alpha}(I); \mathbb{K}).
\end{array} \tag{1} \label{1}
\]

Using this formula, we are able to prove the following lemma. This lemma will be used later in the study of the depth of symbolic powers.

\begin{prop} \label{main}
Let $I$ and $J$ be monomial ideals in $S$. Assume that there exists a function $\varphi: {\rm Mon}(S)\rightarrow{\rm Mon}(S)$ which satisfies the following conditions.
\begin{itemize}
\item[(i)] For every monomial $\mathbf{x}^{\alpha}\in S$ and every subset $F\subseteq [n]$, $\mathbf{x}^{\alpha}\in IS_F$ if and only if $\varphi(\mathbf{x}^{\alpha})\in JS_F$.
\item[(ii)] For every monomial $\mathbf{x}^{\alpha}\in S$, we have ${\rm Supp}(\beta)\subseteq {\rm Supp}(\alpha)$, where $\beta$ is the exponent vector of $\varphi(\mathbf{x}^{\alpha})$, i.e., $\mathbf{x}^{\beta}=\varphi(\mathbf{x}^{\alpha})$.
\end{itemize}
Then$${\rm depth}(S/I)\geq {\rm depth}(S/J).$$
\end{prop}

\begin{proof}
Set $t={\rm depth}(S/I)$. It follows that there exists a vector $\alpha\in \mathbb{Z}^n$ such that $H_{\mathfrak{m}}^t(S/I)_{\alpha}\neq 0$. Thus, equality (\ref{1}), implies that$$\widetilde{H}_{t-\mid{\rm CoSupp}(\alpha)\mid-1}(\Delta_{\alpha}(I); \mathbb{K})\neq 0.$$We write $\alpha=\alpha_+-\alpha_-$ where $\alpha_+$ and $\alpha_-$ are vectors in $\mathbb{Z}_{\geq 0}^n$ with disjoint supports. Assume that $\varphi(\mathbf{x}^{\alpha_+})=\mathbf{x}^{\gamma}$ and set $\beta=\gamma-\alpha_-$. We know from the assumptions that ${\rm Supp}(\gamma)\subseteq {\rm Supp}(\alpha_+)$ and therefore,$${\rm CoSupp}(\beta)={\rm Supp}(\alpha_-)={\rm CoSupp}(\alpha).$$On the other hand, by condition (i),
\begin{align*}
\Delta_{\alpha}(I) & = \{F\subseteq [n]\setminus {\rm CoSupp}(\alpha) : \mathbf{x}^{\alpha}\notin IS_{F\cup {\rm CoSupp}(\alpha)}\}\\ & =\{F\subseteq [n]\setminus {\rm CoSupp}(\alpha) : \mathbf{x}^{\alpha_+}\notin IS_{F\cup {\rm CoSupp}(\alpha)}\}\\ & = \{F\subseteq [n]\setminus {\rm CoSupp}(\beta) : \mathbf{x}^{\gamma}\notin JS_{F\cup {\rm CoSupp}(\beta)}\}\\ & = \{F\subseteq [n]\setminus {\rm CoSupp}(\beta) : \mathbf{x}^{\gamma}\mathbf{x}^{\alpha_-}\notin JS_{F\cup {\rm CoSupp}(\beta)}\}\\ & = \{F\subseteq [n]\setminus {\rm CoSupp}(\beta) : \mathbf{x}^{\beta}\notin JS_{F\cup {\rm CoSupp}(\beta)}\}\\ & = \Delta_{\beta}(J).
\end{align*}

The above equalities together with the fact ${\rm CoSupp}(\beta)={\rm CoSupp}(\alpha)$ implies that$$\widetilde{H}_{t-\mid{\rm CoSupp}(\beta)\mid-1}(\Delta_{\beta}(J); \mathbb{K})\neq 0.$$Hence, it follows from equality (\ref{1}) that $H_{\mathfrak{m}}^t(S/J)_{\beta}\neq 0$ and therefore, ${\rm depth}(S/J)\leq t$.
\end{proof}

In order to use Proposition \ref{main} to compare the depth of symbolic powers, we first need to find a function $\varphi$ which satisfies the assumptions of that proposition. The following lemma tells us how to define this map. We recall that for every monomial $u$ and every integer $1\leq i\leq n$, we denote by ${\rm deg}_{x_i}u$ the highest power of $x_i$ which divides $u$.

\begin{lem} \label{power}
Let $I\subset S$ be a squarefree monomial ideal. Suppose that $m$ and $k$ are positive integers. Then for every integer $j$ with $m-k\leq j\leq m$ and every monomial $u$, we have $u\in I^{(m)}$ if and only if $u^{k+1}\in I^{(km+j)}$.
\end{lem}

\begin{proof}
Let $I=\frak{p}_1\cap\ldots\cap\frak{p}_r,$ be the irredundant primary decomposition of $I$. First suppose that $u\in I^{(m)}$. Then for every integer $t$ with $1\leq t\leq r$, we have $u\in \frak{p}_t^m$. In other words,$$\sum_{x_i\in \frak{p}_t}{\rm deg}_{x_i}u\geq m.$$This implies that$$\sum_{x_i\in \frak{p}_t}{\rm deg}_{x_i}u^{k+1}\geq (k+1)m\geq km+j.$$Hence, $u^{k+1}\in \frak{p}_t^{km+j}$, for every $1\leq t\leq r$. In particular, $u^{k+1}\in I^{(km+j)}$.

Conversely, assume that $u^{k+1}\in I^{(km+j)}$. If $u\notin I^{(m)}$, then there exists an integer $s$ with $1\leq s\leq r$, we have $u\notin \frak{p}_s^m$. In other words,$$\sum_{x_i\in \frak{p}_s}{\rm deg}_{x_i}u\leq m-1.$$This implies that$$\sum_{x_i\in \frak{p}_s}{\rm deg}_{x_i}u^{k+1}\leq (k+1)(m-1)<km+m-k\leq km+j,$$where the last inequality follows from the assumption that $m-k\leq j$. Hence, $u^{k+1}\notin \frak{p}_s^{km+j}$ and consequently, $u^{k+1}\notin I^{(km+j)}$, which is a contradiction.
\end{proof}

We are now ready to reprove the result of Monta\~no and N\'u\~nez-Betancourt.

\begin{thm} \label{depsym}
{\rm (}\cite[Theorem 3.4]{mn}{\rm )} Let $I\subset S$ be a squarefree monomial ideal. Suppose that $m$ and $k$ are positive integers. Then for every integer $j$ with $m-k\leq j\leq m$, we have$${\rm depth}(S/I^{(m)})\geq {\rm depth}(S/I^{(km+j)}).$$
\end{thm}

\begin{proof}
Set $\varphi(u)=u^{k+1}$, for all monomials $u\in {\rm Mon}(S)$. Then for every subset $F\subseteq [n]$ and every monomial $u$, we have $u\in I^{(m)}S_F$ if and only if there is a monomial $v\in \mathbb{K}[x_i \mid i\in F]$ with $uv\in I^{(m)}$. By Lemma \ref{power}, this is equivalent to say that$$\varphi(uv)=u^{k+1}v^{k+1}\in I^{(km+j)},$$which means that $\varphi(u)\in I^{(km+j)}S_F$. The assertion now follows from Proposition \ref{main}.
\end{proof}

As an immediate consequences of Theorem \ref{depsym}, we obtain the following corollaries.

\begin{cor}
For every squarefree monomial ideal $I\subset S$, we have$${\rm depth}(S/I)\geq {\rm depth}(S/I^{(2)})\geq {\rm depth}(S/I^{(3)}).$$
\end{cor}

\begin{cor} \label{semi}
Let $I\subset S$ be a squarefree monomial ideal. Then for every pair of integers $m, k\geq 1$
$${\rm depth} (S/I^{(mk)}) \leq {\rm depth} (S/I^{(m)}).$$
\end{cor}

\begin{proof}
By Theorem \ref{depsym}, $${\rm depth} (S/I^{(m)})\geq {\rm depth}(S/I^{((k-1)m+m)})={\rm depth} (S/I^{(mk)}).$$
\end{proof}

As we mentioned in the introduction, it is known by \cite[Theorem 4.7]{ht} that the sequence $\{{\rm depth}(S/I^{(k)})\}_{k=1}^{\infty}$ is convergent. We also know from \cite[Proposition 2.4]{v1} and \cite[Theorem 2.4]{hktt} that$$\min_k{\rm depth}(S/I^{(k)})=\lim_{k\rightarrow \infty}{\rm depth}(S/I^{(k)})=n-\ell_s(I).$$In the following theorem, we provide alternative proof for theses results. Moreover, we determine an upper bound for ${\rm dstab}_s(I)$ in terms of ${\rm dmin}_s(I)$.

\begin{thm} \label{dmain}
For every squarefree monomial ideal $I\subset S$, the sequence $\{{\rm depth}(S/I^{(k)})\}_{k=1}^{\infty}$ is convergent and$$\min_k{\rm depth}(S/I^{(k)})=\lim_{k\rightarrow \infty}{\rm depth}(S/I^{(k)})= n-\ell_s(I).$$Moreover,$${\rm dstab}_s(I)\leq \max\{1, {\rm dmin}_s(I)^2-{\rm dmin}_s(I)\}.$$
\end{thm}

\begin{proof}
Set $t={\rm dmin}_s(I)$ and suppose that $m={\rm depth}(S/I^{(t)})$. Thus,$$m=\min_k{\rm depth}(S/I^{(k)}).$$If $t=1$, by Corollary \ref{semi}, for every integer $k\geq 1$ we have ${\rm depth}(S/I^{(k)})=m$. Therefore, in this case the sequence $\{{\rm depth}(S/I^{(k)})\}_{k=1}^{\infty}$ is convergent,$$\min_k{\rm depth}(S/I^{(k)})=\lim_{k\rightarrow \infty}{\rm depth}(S/I^{(k)})=m,$$and ${\rm dstab}_s(I)=1$. Now, assume that $t\geq 2$. Again by Corollary \ref{semi}, we have ${\rm depth}(S/I^{(t^2-t)})=m$. For every integer $k > t^2-t$, we write $k=st+j$, where $s$ and $j$ are positive integers and $1\leq j\leq t$. As $k > t^2-t$, we conclude that $s\geq t-1$. It then follows from Theorem \ref{depsym} that$${\rm depth}(S/I^{(k)})={\rm depth}(S/I^{(st+j)})\leq {\rm depth}(S/I^{(t)})=m.$$By the choice of $m$, we conclude that for every integer $k\geq t^2-t$, the equality ${\rm depth}(S/I^{(k)})=m$ holds. Therefore, the sequence $\{{\rm depth}(S/I^{(k)})\}_{k=1}^{\infty}$ is convergent,$$\min_k{\rm depth}(S/I^{(k)})=m=\lim_{k\rightarrow \infty}{\rm depth}(S/I^{(k)})$$ and ${\rm dstab}_s(I)\leq t^2-t$.

It remains to show that$$\lim_{k\rightarrow \infty}{\rm depth}(S/I^{(k)})= n-\ell_s(I).$$By \cite[Theorem 3.2]{hht}, the ring $A=\mathcal{R}_s(I)=\bigoplus_{k=0}^{\infty}I^{(k)}$ is a finitely generated $\mathbb{K}$-algebra. Hence, using \cite[Theorem 2.1]{hht}, we conclude that there exists an integer $c\geq 1$ such that $I^{(ck)}=(I^{(c)})^k$, for every $k\geq 1$. It follows from \cite[Corollary 7.3.15]{v'} that $I^{(c)}$ is a normal ideal and thus, \cite[Proposition 10.3.2]{hh'} implies that$$\lim_{k\rightarrow \infty}{\rm depth}(S/I^{(k)})=\lim_{k\rightarrow \infty}{\rm depth}(S/I^{(ck)})=\lim_{k\rightarrow \infty}{\rm depth}\big(S/(I^{(c)})^k\big)=n-\ell(I^{(c)}).$$Therefore, we only need to show that $\ell(I^{(c)})=\ell_s(I)$. Set$$B=\mathcal{R}_s(I)/\mathfrak{m}\mathcal{R}_s(I)=\bigoplus_
{k=0}^{\infty}I^{(k)}/\mathfrak{m}I^{(k)}$$ and $$B'=\mathcal{R}(I^{(c)})/\mathfrak{m}\mathcal{R}(I^{(c)})=\bigoplus_
{k=0}^{\infty}I^{(ck)}/\mathfrak{m}I^{(ck)}.$$For every integer $k\geq 0$ and every monomial $u\in I^{(k)}$, we have $u^c\in I^{(ck)}$. This shows that $B$ is an integral extension of $B'$. Therefore, using \cite[Theorem 2.2.5]{hs}, we deduce that$$\ell_s(I)={\rm dim} B= {\rm dim} B'=\ell(I^{(c)}).$$This completes the proof.
\end{proof}


\section{Stanley depth of symbolic powers} \label{sec3}

In this section, we study the Stanley depth of symbolic powers of squarefree monomial ideals. The first main result of this section, Theorem \ref{sdepsym}, states that if $m$ and $k$ are positive integers, then for every integer $j$ with $m-k\leq j\leq m$, we have$${\rm sdepth}(I^{(m)})\geq {\rm sdepth}(I^{(km+j)}) \ \ \ \ \ \ \ \ {\rm and} \ \ \ \ \ \ \ \ {\rm sdepth}(S/I^{(m)})\geq {\rm sdepth}(S/I^{(km+j)}).$$In order to prove the above inequalities we need the following result from \cite{s4}.

\begin{lem} \label{scand}
{\rm (}\cite[Theorem 2.1]{s4}{\rm )} Let $I$ and $J$ be monomial ideals in $S$. Assume that there exists a function $\varphi: {\rm Mon}(S)\rightarrow{\rm Mon}(S)$, such that the following conditions are satisfied.
\begin{itemize}
\item[(i)] For every monomial $u\in {\rm Mon}(S)$, $u\in I$ if and only if $\varphi(u)\in J$.
\item[(ii)] For every Stanley space $u\mathbb{K}[Z]\subseteq S$ and every monomial $v\in {\rm Mon}(S)$, we have $v\in u\mathbb{K}[Z]$ if and only if $\varphi(v)\in \varphi(u)\mathbb{K}[Z]$.
\end{itemize}
Then
$${\rm sdepth}(I) \geq {\rm sdepth}(J) \ \ \ \ \ \ \ \ and \ \ \ \ \ \ \ \ {\rm sdepth}(S/I) \geq {\rm sdepth}(S/J).$$
\end{lem}

The following theorem is the first main result of this section. It is a consequence of Lemmata \ref{power} and \ref{scand}.

\begin{thm} \label{sdepsym}
Let $I\subset S$ be a squarefree monomial ideal. Suppose that $m$ and $k$ are positive integers. Then for every integer $j$ with $m-k\leq j\leq m$, we have$${\rm sdepth}(I^{(m)})\geq {\rm sdepth}(I^{(km+j)}) \ \ \ \ \ \ \ \ and \ \ \ \ \ \ \ \ {\rm sdepth}(S/I^{(m)})\geq {\rm sdepth}(S/I^{(km+j)}).$$
\end{thm}

\begin{proof}
Set $\varphi(u)=u^{k+1}$, for all monomials $u\in {\rm Mon}(S)$. By Lemma \ref{power}, for every monomial $u$, we have $u\in I$ if and only if $\varphi(u)\in I^{(km+j)}$. Hence, the assertions follow from Lemma \ref{scand}.
\end{proof}

As an immediate consequences of Theorem \ref{sdepsym}, we obtain the following corollaries.

\begin{cor}
For every squarefree monomial ideal $I\subset S$, we have$${\rm sdepth}(S/I)\geq {\rm sdepth}(S/I^{(2)})\geq {\rm sdepth}(S/I^{(3)})$$and$${\rm sdepth}(I)\geq {\rm sdepth}(I^{(2)})\geq {\rm sdepth}(I^{(3)}).$$
\end{cor}

In the following theorem, which is the second main result of this section, we prove that the Stanley depth of high symbolic powers of a squarefree monomial ideal is constant. Moreover, we show that this constant number is equal to the minimum value of the Stanley depth of symbolic powers.

\begin{thm} \label{sdmain}
For every squarefree monomial ideal $I$, the sequences $\{{\rm sdepth}(S/I^{(k)})\}_{k=1}^{\infty}$ and $\{{\rm sdepth}(I^{(k)})\}_{k=1}^{\infty}$ are convergent. Moreover,$$\min_k{\rm sdepth}(S/I^{(k)})=\lim_{k\rightarrow \infty}{\rm sdepth}(S/I^{(k)}),$$and$$\min_k{\rm sdepth}(I^{(k)})=\lim_{k\rightarrow \infty}{\rm sdepth}(I^{(k)}).$$
\end{thm}

\begin{proof}
Set$$m:=\min_k{\rm sdepth}(S/I^{(k)}),$$and let $t\geq 1$ be the smallest integer with ${\rm sdepth}(S/I^{(t)})=m$. If $t=1$, by \cite[Theorem 3.1]{s4}, for every integer $k\geq 1$ we have ${\rm sdepth}(S/I^{(k)})=m$. Therefore, in this case the sequence $\{{\rm sdepth}(S/I^{(k)})\}_{k=1}^{\infty}$ is convergent and$$\min_k{\rm sdepth}(S/I^{(k)})=\lim_{k\rightarrow \infty}{\rm sdepth}(S/I^{(k)})=m.$$Now, assume that $t\geq 2$. Again by \cite[Theorem 3.1]{s4}, we have ${\rm sdepth}(S/I^{(t^2-t)})=m$. For every integer $k > t^2-t$, we write $k=st+j$, where $s$ and $j$ are positive integers and $1\leq j\leq t$. As $k > t^2-t$, we conclude that $s\geq t-1$. It then follows from Theorem \ref{sdepsym} that$${\rm sdepth}(S/I^{(k)})={\rm sdepth}(S/I^{(st+j)})\leq {\rm sdepth}(S/I^{(t)})=m.$$By the choice of $m$, we conclude that for every integer $k\geq t^2-t$, the equality ${\rm sdepth}(S/I^{(k)})=m$ holds. Therefore, the sequence $\{{\rm sdepth}(S/I^{(k)})\}_{k=1}^{\infty}$ is convergent and$$\min_k{\rm sdepth}(S/I^{(k)})=m=\lim_{k\rightarrow \infty}{\rm sdepth}(S/I^{(k)}).$$

By using a similar argument, one can prove that the sequence $\{{\rm sdepth}(I^{(k)})\}_{k=1}^{\infty}$ is convergent and$$\min_k{\rm sdepth}(I^{(k)})=\lim_{k\rightarrow \infty}{\rm sdepth}(I^{(k)}).$$
\end{proof}

As an immediate consequence of the proof of Theorem \ref{sdmain}, we obtain the following result, which provides an upper bound for ${\rm sdstab}_s(S/I)$ (resp. ${\rm sdstab}_s(I)$) in terms of ${\rm sdmin}_s(S/I)$ (resp. ${\rm sdmin}_s(I)$).

\begin{cor} \label{sstabmin}
For every squarefree monomial ideal $I\subset S$, we have$${\rm sdstab}_s(I)\leq \max\{1, {\rm sdmin}_s(I)^2-{\rm sdmin}_s(I)\}$$and$${\rm sdstab}_s(S/I)\leq \max\{1, {\rm sdmin}_s(S/I)^2-{\rm sdmin}_s(S/I)\}.$$
\end{cor}

Inspired by Theorem \ref{sdmain}, it is natural to ask about the limit values of the sequences $\{{\rm sdepth}(S/I^{(k)})\}_{k=1}^{\infty}$ and $\{{\rm sdepth}(I^{(k)})\}_{k=1}^{\infty}$. In general, we do not know how to compute these limits. However, in Theorem \ref{mat}, we investigate this question for the case that $I$ is the Stanley-Reisner ideal of a matroid. We first need the following lemma.

\begin{lem} \label{colo}
Let $I\subset S$ be an unmixed squarefree monomial ideal of height $h$. Then for every $k\geq 1$, we have
\begin{itemize}
\item[(i)] $(I^{(k)}: x_1x_2\ldots x_n)=S$, if $k\leq h$, and
\item[(ii)] $(I^{(k)}: x_1x_2\ldots x_n)=I^{(k-h)}$, if $k> h$.
\end{itemize}
\end{lem}

\begin{proof}
Let $I=\frak{p}_1\cap\ldots\cap\frak{p}_r,$ be the irredundant primary decomposition of $I$. Then$$(I^{(k)}: x_1x_2\ldots x_n)=\bigcap_{i=1}^r(\frak{p}_i^k: x_1x_2\ldots x_n).$$The assertions now follow from the fact the for every monomial prime ideal $\frak{p}$ of height $h$, we have $(\frak{p}^k: x_1x_2\ldots x_n)=S$, if $k\leq h$, and $(\frak{p}^k: x_1x_2\ldots x_n)=\frak{p}^{(k-h)}$, if $k> h$.
\end{proof}

In the following theorem, we study the Stanley depth of high symbolic powers of the Stanley-Reisner ideal of a matroid.

\begin{thm} \label{mat}
Let $\Delta$ be a matroid. Then$$\lim_{k\rightarrow \infty}{\rm sdepth}(S/I_{\Delta}^{(k)})=n-\ell_s(I_{\Delta})=\dim \Delta+1$$and$$\lim_{k\rightarrow \infty}{\rm sdepth}(I_{\Delta}^{(k)})\geq n-\ell_s(I_{\Delta})+1.$$
\end{thm}

\begin{proof}
By \cite[Theorem 2.1]{v1} (see also \cite[Theorem 1.1]{tt1}), we know that $S/I_{\Delta}^{(k)}$ is Cohen-Macaulay, for any integer $k\geq 1$. Hence, by Theorem \ref{dmain}, we have$$n-\ell_s(I_{\Delta})=\lim_{k\rightarrow \infty}{\rm depth}(S/I_{\Delta}^{(k)})=\lim_{k\rightarrow \infty}\dim S/I_{\Delta}^{(k)}=\dim S/I_{\Delta}=\dim \Delta+1.$$Also, we know from \cite[Theorem 2.1]{bs} that for every integer $k\geq 1$, the module $S/I_{\Delta}^{(k)}$ is a clean module. Thus, by \cite[Page 14]{h}, we have$${\rm sdepth}(S/I_{\Delta}^{(k)})={\rm depth}(S/I_{\Delta}^{(k)}),$$for every integer $k\geq 1$. This implies that$$\lim_{k\rightarrow \infty}{\rm sdepth}(S/I_{\Delta}^{(k)})=n-\ell_s(I_{\Delta})=\dim \Delta+1.$$

In order to complete the proof, it is enough to prove that
\[
\begin{array}{rl}
{\rm sdepth}(I_{\Delta}^{(k)})\geq \dim \Delta+2,
\end{array}
\]
for every integer $k\geq 1$.

Assume that $V(\Delta)=[n]$. We prove the above inequality by induction on $n+k$. We first note that if $\Delta$ is a simplex, then $I_{\Delta}=0$ and thus, ${\rm sdepth}(I_{\Delta}^{(k)})=\infty\geq \dim \Delta+2$, for every $k\geq 1$. In particular, the assertion is trivial for $n=1$. Suppose $k=1$. It is well-known that $\Delta$ is a vertex decomposable simplicial complex (see e.g., \cite[Page 2256]{nr}). Hence, using \cite[Theorem 4.2]{ks}, we conclude that$${\rm sdepth}(I_{\Delta})\geq {\rm depth}(I_{\Delta})={\rm depth}(S/I_{\Delta})+1=\dim S/I_{\Delta}+1=\dim \Delta+2.$$Thus, suppose that $n,k\geq 2$ and $\Delta$ is not a simplex.

For every integer $i$ with $i\leq i\leq n$, let $S_i=\mathbb{K}[x_1, \ldots, x_{i-1}, x_{i+1}, \ldots, x_n]$ be the polynomial ring obtained from $S$ by deleting the variable $x_i$ and set $\Gamma_i:={\rm del}_{\Delta}(i)$. If $\Delta$ has a vertex $i$ which belongs to every facet of $\Delta$, then $I_{\Delta}=I_{\Gamma_i}S$. Since $\Gamma_i={\rm del}_{\Delta}(i)$ is a matroid, by \cite[Lemma 3.6]{hvz} and the induction hypothesis, we conclude that
\begin{align*}
{\rm sdepth}(I_{\Delta}^{(k)}) & ={\rm sdepth}(I_{\Gamma_i}^{(k)}S)={\rm sdepth}_{S_i}(I_{\Gamma_i}^{(k)}S_i)+1 \geq \dim \Gamma_i+2+1\\ & =\dim \Delta-1+2+1=\dim \Delta+2.
\end{align*}
Hence, we assume the $\Delta$ has no vertex which belongs to every facet of $\Delta$. This implies that ${\rm dim}\Gamma_i={\rm dim}\Delta$, for every $i\in [n]$.

Consider the ideals $I_1=I_{\Delta}^{(k)}\cap S_1$ and $I_1'=(I_{\Delta}^{(k)}:x_1)$. Now $I_{\Delta}^{(k)}=I_1\oplus x_1I_1'$ and therefore by definition of  the Stanley depth we have
\[
\begin{array}{rl}
{\rm sdepth}(I_{\Delta}^{(k)})\geq \min \{{\rm sdepth}_{S_1}(I_1), {\rm sdepth}_S(I_1')\}.
\end{array} \tag{2} \label{3}
\]

Notice that $I_1=(I_{\Delta}\cap S_1)^{(k)}=I_{\Gamma_1}^{(k)}S_1$. Thus, it follows from the induction hypothesis that
$${\rm sdepth}_{S_1}(I_1)\geq \dim \Gamma_1+2=\dim \Delta+2.$$
Hence, using the inequality (\ref{3}), it is enough to prove that ${\rm sdepth}_S(I_1')\geq \dim \Delta+2$.

For every integer $i$ with $2\leq i\leq n$, consider the ideals $I_i'=(I_{i-1}':x_i)$ and $I_i=I_{i-1}'\cap S_i$.

{\bf Claim.} For every integer $i$ with $1\leq i\leq n-1$, we have$${\rm sdepth}(I_i')\geq \min \{\dim \Delta+2, {\rm sdepth}(I_{i+1}')\}.$$

{\it Proof of the Claim.} For every integer $i$ with $1\leq i\leq n-1$, we have $I_i'=I_{i+1}\oplus x_{i+1}I_{i+1}'$ and therefore by definition of  the Stanley depth we have
\[
\begin{array}{rl}
{\rm sdepth}(I_i')\geq \min \{{\rm sdepth}_{S_{i+1}}(I_{i+1}), {\rm sdepth}_S(I_{i+1}')\}.
\end{array} \tag{3} \label{4}
\]

Notice that for every integer $i$ with $1\leq i\leq n-1$, we have $I_i'=(I_{\Delta}^{(k)}:x_1x_2\ldots x_i)$. Thus,$$I_{i+1}=I_i'\cap S_{i+1}=((I_{\Delta}^{(k)}\cap S_{i+1}):_{S_{i+1}}x_1x_2\ldots x_i).$$Hence, it follows from \cite[Proposition 2]{p} (see also \cite[Proposition 2.5]{s4}) that

\[
\begin{array}{rl}
{\rm sdepth}_{S_{i+1}}(I_{i+1})\geq {\rm sdepth}_{S_{i+1}}(I_{\Delta}^{(k)}\cap S_{i+1}).
\end{array} \tag{4} \label{5}
\]

On the other hand,$$I_{\Delta}^{(k)}\cap S_{i+1}=(I_{\Delta}\cap S_{i+1})^{(k)}=I_{\Gamma_{i+1}}^{(k)}S_{i+1}.$$Thus, it follows from the induction hypothesis that$${\rm sdepth}_{S_{i+1}}(I_{\Delta}^{(k)}\cap S_{i+1})\geq \dim \Gamma_{i+1}+2=\dim \Delta+2.$$Therefore, the claim follows by inequalities (\ref{4}) and (\ref{5}).

Set $d=\dim \Delta+1$. Since $\Delta$ is pure, we conclude from \cite[Lemma 1.5.4]{hh'} that $I_{\Delta}$ is an unmixed ideal of height $n-d$. We know that $I_n'=(I_{\Delta}^{(k)}:x_1x_2\ldots x_n)$ and hence, Lemma \ref{colo} implies that either $I_n'=S$ or $I_n'=I_{\Delta}^{(k-n+d)}$. In the first case, we have$${\rm sdepth}(I_n')=n\geq \dim \Delta+2,$$(where the inequality follows from the assumption that $\Delta$ is not a simplex) and in the second case, by the induction hypothesis we have$${\rm sdepth}(I_n')\geq \dim \Delta+2.$$Thus, in both cases we have ${\rm sdepth}(I_n')\geq \dim \Delta+2$. Using the claim repeatedly, we conclude that ${\rm sdepth}(I_1')\geq \dim \Delta+2$. This completes the proof of the theorem.
\end{proof}





\end{document}